\documentclass[preprint,number,12pt,times]{elsarticle}
\usepackage{pstricks}
\usepackage{graphicx, pst-plot, pst-node, pst-text, pst-tree}
\usepackage{amsfonts}
\usepackage{hyperref}
\usepackage{textcomp}
\usepackage{amsmath}
\usepackage{amssymb}
\usepackage{color}
\hypersetup{
        colorlinks=true,
        linkcolor=darkblue,
        anchorcolor=darkblue,
        citecolor=darkblue,
        urlcolor=darkblue,
        pdfpagemode=UseOutlines,
        pdftitle={Growth rates for subclasses of Av(321)},
        pdfsubject={Combinatorics},
        pdfauthor={M. H. Albert, M. D. Atkinson, R. Brignall, N. Ruskuc, Rebecca Smith and J. West},
}
\newtheorem{theorem}{Theorem}
\newtheorem{lemma}[theorem]{Lemma}
\newtheorem{conjecture}[theorem]{Conjecture}
\newtheorem{proposition}[theorem]{Proposition}
\newtheorem{corollary}[theorem]{Corollary}
\newtheorem{observation}[theorem]{Observation}

\newtheorem{example}{Example}
\newproof{proof}{Proof}
\newproof{proof-main}{Proof of Theorem~\ref{THM Main Theorem}}

\newcommand{\A}{\mathcal{A}}
\newcommand{\B}{\mathcal{B}}
\newcommand{\C}{\mathcal{C}}
\newcommand{\D}{\mathcal{D}}

\newcommand{\I}{\mathcal{I}}
\newcommand{\M}{\mathcal{M}}
\renewcommand{\S}{\mathcal{S}}

\newcommand{\avoid}{\mathrm{Av}}
\newcommand{\rigidReduction}{\mathrm{red}}
\newcommand{\involvedIn}{\preceq}
\newcommand{\stanleyWilfLimit}{s}

\begin{document}
\begin{frontmatter}

\title{Growth rates for subclasses of $\avoid (321)$.}

\author[otago]{M.~H.~Albert}
\ead{malbert@cs.otago.ac.nz}
\author[otago]{M.~D.~Atkinson}
\ead{mike@cs.otago.ac.nz}
\author[bristol]{R.~Brignall\fnref{heilb}}
\ead{robert.brignall@bris.ac.uk}
\author[sta]{N.~Ru\v{s}kuc}
\ead{nik@mcs.st-and.ac.uk}
\author[suny]{Rebecca~Smith}
\ead{rebecca@brockport.edu}
\author{J.~West}
\ead{julian@mcs.st-and.ac.uk}
\fntext[heilb]{Supported by the Heilbronn Institute for Mathematical Research.}
\address[otago]{Department of Computer Science, University of Otago} 
\address[bristol]{Department of Mathematics, University of Bristol}
\address[sta]{School of Mathematics and Statistics, University of St Andrews}
\address[suny]{Department of Mathematics, SUNY Brockport}

\begin{abstract}
Pattern classes which avoid $321$ and other patterns are shown to have the same growth rates as similar (but strictly larger) classes obtained by adding articulation points to any or all of the other patterns. The method of proof is to show that the elements of the latter classes can be represented as \emph{bounded merges} of elements of the original class, and that the bounded merge construction does not change growth rates.
\end{abstract}

\end{frontmatter}

\section{Introduction}

A pattern class is, roughly,  a collection of (finite) permutations that satisfy certain restrictions on the configurations of their elements (formal definitions can be found in the next section). For example, the collection of all permutations containing no descending subsequence of length 3 is such a class. In general to denote that a pattern class $\C$ is determined by a set of restrictions $B$ we write $\C=\avoid(B)$.  The study of such classes dates back at least to work of Knuth \cite{Knuth75}, or even further to the celebrated result of Erd\H{o}s and Szekeres \cite{ES35} that every permutation of length greater than $ad$ must include either an ascending subsequence of length $a+1$ or a descending one of length $d+1$.

Initially, research into pattern classes focussed on enumeration -- determining the number of permutations of length $n$ in a given pattern class.  An early result of this type \cite{Knuth75} was that $\avoid(231)$ and $\avoid(321)$ are both enumerated by the Catalan sequence (and by easy symmetries so also is every class $\avoid(\alpha)$ with $|\alpha|=3$).  Early hopes that $\avoid(231)$ and $\avoid(321)$ might have further properties in common have largely foundered since the discovery \cite{AMR02a} that $\avoid(231)$ contains only  countably many subclasses whilst $\avoid(321)$ contains uncountably many.  In fact $\avoid(231)$ is a very tractable class compared to $\avoid(321)$ and, in particular, there is an efficient algorithm \cite{AA03b} to enumerate $\avoid(B)$ whenever $231\in B$.  By contrast the subclasses of $\avoid(321)$ are generally impossible to enumerate exactly and so attention has turned to growth rate estimates.

Growth rate estimates have become an important way of approximating the number of permutations in a pattern class since Marcus and Tardos~\cite{MT04} proved the Stanley-Wilf conjecture that for every proper pattern class there is an exponential bound on the number of permutations of length $n$ which it contains.  Their result implies that every proper pattern class $\C$ has a \emph{growth rate} defined to be the limit superior of the $n^{\mbox{\scriptsize th}}$ root of the number of permutations in $\C$ of length $n$. Growth rates have been investigated by B\'{o}na \cite{Bona05, Bona07} who found bounds (relative to the size of the forbidden patterns) and results on what form this growth rate might take. Recently, Vatter \cite{Vatter08} has proven  that every real number greater than 2.482 occurs as the growth rate of some pattern class. Because of these results and others we shall investigate the growth rates of pattern subclasses of $\avoid(321)$
and particularly when distinct subclasses of $\avoid(321)$  have the same growth rate. 


Consider a pattern class $\C$ of the form $\avoid(321, X)$ where $X$ is some arbitrary set of permutations. Consider also $\C'=\avoid(321, X')$ where $X'$ is obtained from $X$ by adding or removing ``articulation points'' (similar to the 3 of 21354) anywhere within the patterns of $X$. The main result of this paper is that $\C$ and $\C'$  have the same growth rate. In order to prove this result we introduce a number of new concepts and constructions, including the notions of $k$-rigidity, bounded merges, and staircase decompositions, which we discuss in some generality.

The structure of the remainder of this paper is as follows: 
\begin{description}
\item[Section 2] introduces the formal definitions, and certain preliminary results concerning rigidity and growth rates.
\item[Section 3] contains the proof of the main result, divided into two cases for clarity, using staircase decompositions.
\item[Section 4] examines the distributive lattices of occurrences of 21 in a 321-avoiding permutation, and shows that every subdirect product of two chains can arise in this fashion.
\item[Section 5] concludes the paper with some further remarks, and open problems.
\end{description}

\section{Preliminaries}

A permutation $\pi \in \S_n$ is a bijective map from $[n] = \{1,2,\ldots,n\}$ to itself, and is therefore a set of ordered pairs 
\[
\{(1, \pi(1)), (2, \pi(2)), \ldots, (n, \pi(n))\}
\]
(traditionally and more frequently written as the sequence $\pi(1)\pi(2)\cdots\pi(n)$).
So, when we say $x \in \pi$ we are simply referring to some member of this set. However, it is frequently necessary to relate elements of $\pi$ either by the values of their first or second coordinates. Normally, we think of the first coordinates as lying on a horizontal axis so words and phrases such as ``precedes'', ``follows'', ``to the left of'', etc.\ refer to that ordering. Conversely words such as ``larger'', ``smaller'', ``above'' and ``below'' relate to the ordering of the second coordinate.

An \emph{involvement} or \emph{embedding} of a permutation $\alpha$ in $\pi$ is a map $f: \alpha \to \pi$ that respects both these orderings. In other words $x$ precedes (is larger than) $y$ in $\alpha$ if and only if $f(x)$ precedes (is larger than) $f(y)$ in $\pi$. In particular an embedding is necessarily injective. The composition of embeddings is an embedding and so the relation ``is involved in'' is a partial order, which will be denoted $\involvedIn$. If a subset of $\pi$ is the image of $\alpha$ under an embedding, then we say that the pattern of the subset is $\alpha$. We say that $x \in \pi$ \emph{occurs as an $i$ in an embedding of $\alpha$} (or just ``as $i$ in an $\alpha$'') if there is an embedding of $\alpha$ in $\pi$ such that $x$ is the image of the element of $\alpha$ whose \emph{second} coordinate (i.e.\ value\footnote{Why value? Because, in the usual ``one line'' notation for permutations, it is easy to identify the element of value $i$, and not necessarily so easy to identify the element at position $i$.}) is $i$. A \emph{pattern class}, or simply \emph{class} of permutations is a set of permutations closed downward under $\involvedIn$. Such a class, $\C$, can also be defined as the set of permutations which \emph{avoid}, i.e.\ do not involve, any of the elements of some set $B$ of permutations. In that case we write $\C = \avoid(B)$. If $B$ is a $\involvedIn$-antichain, then it is called the \emph{basis} of $\C$ (note that, for any set $B$, the set of minimal elements
of $B$ is an antichain and forms the basis of $\avoid(B)$). We define  the \emph{growth rate} (sometimes called the \emph{Stanley-Wilf limit}, or \emph{upper growth rate}) of $\C$:
\[
\stanleyWilfLimit (\C) = \limsup_{n \to \infty} | \C \cap \S_n |^{1/n}.
\]
As noted in the introduction, Marcus and Tardos \cite{MT04} proved that if $\C$ is a proper pattern class, then $\stanleyWilfLimit ( \C ) < \infty$.

The \emph{increasing} and \emph{decreasing} permutations of length $k$ are 
\[
\begin{array}{lcl}
\iota_k &=& \{(1,1), (2,2), (3,3), \ldots, (k,k)\} \\ 
\delta_k &=& \{(k,1), (k-1,2), (k-2,3), \ldots, (1,k)\}\\ 
\end{array}
\]
respectively. A subset of $\pi$ is called increasing (respectively decreasing) if its pattern is some increasing (decreasing) permutation.

Throughout this paper, we are primarily concerned with permutations that can be written as the union of $k$ increasing subsets for some fixed value of $k$. These permutations form a pattern class $\I_k$, whose basis is the single decreasing permutation $\delta_{k+1}$.
We say that a permutation $\pi \in \I_k$ is \emph{$k$-rigid} if every element of $\pi$ belongs to a subset whose pattern is $\delta_k$.

Suppose that $\pi \in \I_k$. We can define a decomposition of $\pi$ into increasing subsets $C_1$, $C_2$, \ldots, $C_k$ by defining, for $1 \leq t \leq k$:
\[
C_t = \left\{ x \in \pi \, : \, \parbox{6cm}{$x$ occurs as the maximum of some $\delta_t$ but not of any $\delta_{t+1}$} \right\}.
\]
This decomposition is the one produced by a greedy algorithm, which takes the elements of $\pi$ in order from right to left, and adds each successive element $x$ to the first $C_j$ of which $x$ is smaller than the current minimum. If $x \in \pi$ belongs to $C_i$ then we say that the \emph{rank} of $x$ is $i$.

\begin{lemma}
If $\pi \in \I_k$, and $x \in \pi$ occurs as an $i$ in some $\delta_k$, then the rank of $x$ is $i$. Consequently, the position of $x$ in all the $\delta_k$ to which it belongs is the same.
\end{lemma}

\begin{proof}
Choose a $\delta_k$ in which $x$ occurs as $i$, and write it in one line notation as $A x B$ (so $A$ is a decreasing sequence of length $k-i$ and $B$ a decreasing sequence of length $i-1$). Then $x$ occurs as the maximum of the $\delta_i$, $x B$. It cannot occur as the maximum of any $\delta_{i+1}$, $x C$, because then $A x C$ would be a $\delta_{k+1}$ in $\pi$.
\end{proof}

It follows that if $\rho$ is $k$-rigid, then any embedding of $\rho$ in $\pi \in \I_k$ must preserve the ranks of the elements of $\rho$, as it preserves sets whose pattern is $\delta_k$.

If two elements of a permutation coincide or form a $12$ pattern, then it makes sense to speak of their \emph{infimum} -- it is simply the smaller and earlier of the two, and likewise their \emph{supremum} which is the larger and later. If $f, g : \rho \to \pi$ are two embeddings of a $k$-rigid permutation into an element of $\I_k$, then for any $x \in \rho$, the ranks of $f(x)$ and $g(x)$ are the same. Therefore $f(x)$ and $g(x)$ occur in some increasing subset of $\pi$ and hence their infimum and supremum are defined. In fact more is true:

\begin{theorem}
\label{THM Lattice Structure}
Let $\pi \in \I_k$, $\rho$ a $k$-rigid permutation, and two embeddings $f, g : \rho \to \pi$ be given. Then $I, S : \rho \to \pi$ defined for $x \in \rho$ by $I(x) = \inf (f(x), g(x))$, and  $S(x) = \sup (f(x), g(x))$ are also embeddings of $\rho$ in $\pi$. In particular, the embeddings of $\rho$ in $\pi$ form a distributive lattice.
\end{theorem}

\begin{proof}
We give the argument for $I$ only (that for $S$ is similar). It suffices to show that for any $x, y \in \rho$ (without loss of generality, $x$ preceding $y$), the pattern of $I(x)$ and $I(y)$ in $\pi$ is the same as the pattern of $x$ and $y$ in $\rho$. But, this is essentially trivial. If the pattern of $xy$ is $12$ then $\inf (f(x), f(y)) = f(x)$ and $\inf(g(x), g(y)) = g(x)$. So, $\inf (f(x), g(x))$ must form a 12 pattern with $\inf(f(y), g(y))$.  The case where $xy$ has pattern $21$ is just the same.
\end{proof}

More generally, given two embeddings $f$ and $g$ of an arbitrary permutation $\alpha$ in an arbitrary permutation $\beta$ such that the images $f(a)$ and $g(a)$ of any $a \in \alpha$ coincide or form a $12$ pattern, the maps $I$ and $S$ defined in the theorem are also embeddings of $\alpha$ in $\beta$. We will defer a discussion of the distributive lattices mentioned in the theorem above to Section \ref{aftermath}.

Applying the previous theorem repeatedly, we can take the infimum of \emph{all} of the embeddings of a $k$-rigid permutation into an element $\pi \in \I_k$, thus obtaining:

\begin{corollary}
\label{COR leftmost-bottommost}
Let $\pi \in \I_k$ and $\rho$ a $k$-rigid permutation be given. If $\rho \involvedIn \pi$ then there is an embedding of $\rho$ in $\pi$ which simultaneously minimizes the position and value of every element of the image of $\rho$ among all such embeddings.
\end{corollary}

Naturally enough, we call the embedding whose existence is asserted by this corollary the \emph{leftmost-bottommost} embedding of $\rho$ in $\pi$.

A permutation $\pi$ is called a \emph{merge} of two permutations $\alpha$ and $\beta$ if it can be written as the disjoint union of two sets, the first of which has pattern $\alpha$ and the second of which has pattern $\beta$. If $\A$ and $\B$ are pattern classes, then
\[
\M(\A, \B) = \{ \pi \, : \, \mbox{$\pi$ is a merge of some $\alpha \in \A$ and some $\beta \in \B$} \}
\]
is also a permutation class, called the \emph{merge} of $\A$ and $\B$. For instance $\M(\I_s, \I_t) = \I_{s+t}$ for any $s$ and $t$.

Let two permutations $\alpha$ and $\beta$ be given, together with embeddings $a : \alpha \to \pi$, $b : \beta \to \pi$ that witness $\pi$ being a merge of $\alpha$ and $\beta$ (so the ranges of the embeddings are disjoint and their union is equal to $\pi$). For $x \in \pi$ define the \emph{type of $x$}, $\mathrm{tp}(x) = a$ if $x$ is in the range of $a$ and $\mathrm{tp}(x) = b$ if it is in the range of $b$. For $1 \leq c < |\pi|$, if the types of $(c, \pi(c))$ and $(c+1, \pi(c+1))$ are different, then we say that there is a \emph{type change by position} at $c$. Similarly, for $1 \leq r < |\pi|$, if the types of $(\pi^{-1}(r), r)$ and $(\pi^{-1}(r+1), r+1)$ are different, then we say that there is a \emph{type change by value} at $r$.

Given a positive integer $B$ and two permutation classes $\C$ and $\D$ we define the \emph{$B$-bounded merge of $\C$ and $\D$}:
\[
\M_{B}(\C, \D) =  \left\{ \pi \, : \, \parbox{9cm}{$\pi$ is a merge of some $\alpha \in \C$ and some $\beta \in \D$ having at most $B$ type
changes in total, either by position or value} \right\}
\]
As the number of type changes cannot increase when we delete elements of a merge, $\M_B(\C, \D)$ is also a permutation class.

\begin{example}
The permutation 
\[\{(1,1), (2,2), (3,3), (4,7), (5,8), (6,9), (7,4), (8,5), (9,6)\}\] 
($123789456$ in one line notation) lies in $\M_3(\I_1,\I_1)$ because of the subsequences $123789$ and $456$ and the type changes $(6,9)$ to $(7,4)$ by position and $(3,3)$ to $(7,4)$ and $(9,6)$ to $(4,7)$ by value.
\end{example}

\begin{theorem}
\label{THM Merge Growth}
Let a positive integer $B$ and two permutation classes $\C$ and $\D$ be given. Then,
\begin{eqnarray*}
\stanleyWilfLimit(\M(\C, \D)) &\leq& \left( \sqrt{\stanleyWilfLimit(\C)} + \sqrt{\stanleyWilfLimit(\D)} \right)^2, \mbox{and} \\
\stanleyWilfLimit(\M_B (\C, \D)) &=& \max (\stanleyWilfLimit(\C) , \stanleyWilfLimit(\D) ).
\end{eqnarray*}
\end{theorem}

\begin{proof}
Let $c_n = |\C \cap \S_n|$, $d_n = | \D \cap \S_n |$, $M_n = | \M(\A, \B) \cap \S_n |$ and $m_n = | \M_B(\A, \B) \cap \S_n |$. A merge of $\alpha \in \A \cap \S_k$ and $\beta \in \B \cap \S_{n-k} $ can be defined by independently choosing $k$ (from $n$) positions and $k$ values to hold the pattern $\alpha$, while fitting the pattern $\beta$ in the remaining positions and values. It follows that:
\[
M_n \leq \sum_{k=0}^n \binom{n}{k}^2 c_k d_{n-k}.
\]
So,
\[
s(\M(\C, \D)) \leq \limsup_{n \to \infty} \left( \sum_{k=0}^n \binom{n}{k}^2 c_k d_{n-k} \right)^{1/n}.
\]
The similarity of the square root of each term in the sum to a term of the expansion of $\left( \sqrt{s(\C)} + \sqrt{s(\D)} \right)^n$ is sufficient to establish the first of the results claimed in the theorem (an argument that goes back to \cite{Regev}).

For the second result, in order to specify a $B$-bounded merge of length $n$ we need to specify at most $B$ positions and values at which a type change can occur, and then two permutations in $\C$ and $\D$ of suitable length. Additionally, $\C, \D \subseteq \M_B(\C, \D)$. So (certainly for $n > 2B$):
\[
\max( c_n, d_n) \leq m_n \leq \binom{n}{B}^2 \max \{ c_k d_{n-k} \, : \, 0 \leq k \leq n \}.
\]
Taking $n^{\mbox{\scriptsize th}}$ roots throughout, and observing that $\binom{n}{B}^{2/n} \to 1$ establishes the second result.
\end{proof}

Note that $\stanleyWilfLimit(\I_k) = k^2$, so the bound given by the first estimate is tight for $\M(\I_n, \I_m)$. For the remainder of this paper we will only be using the second of these estimates; that the growth rate of a bounded merge of two permutation classes is the maximum of their individual growth rates.

The \emph{direct sum} $\alpha \oplus \beta$ of two permutations $\alpha$ and $\beta$ is that merge of $\alpha$ with $\beta$ in which the image of $\alpha$ occupies the first $|\alpha|$ places both by position and value. A permutation $\pi$ is called \emph{plus indecomposable} if $\pi \neq \alpha \oplus \beta$ for any pair of non-empty permutations $\alpha$ and $\beta$. 

If $\pi \in \I_2$ is not 2-rigid, then, for some $\alpha$ and $\beta$, $\pi = \alpha \oplus 1 \oplus \beta$ since it must contain an element which has no larger preceding element, nor any smaller following element. Thus, all the preceding elements (of pattern $\alpha$) are smaller than it and the following ones (of pattern $\beta$) are larger. Such an element is called an \emph{articulation point} of $\pi$. Conversely, $\pi \in \I_2$ is 2-rigid exactly if $\pi = \alpha_1 \oplus \alpha_2 \oplus \cdots \oplus \alpha_k$ where $k \geq 1$ and each $\alpha_i$ is a plus indecomposable permutation of length at least 2.

Let $1^n = \iota_n$ be the direct sum of $n$ copies of the singleton permutation. If $\pi \in \I_2$ is an arbitrary permutation then there is a unique sequence $\rho_1$, $\rho_2$, \ldots, $\rho_c$ of plus indecomposable permutations all of length at least 2 such that:
\[
\pi = 1^{m_0} \oplus \rho_1 \oplus 1^{m_1} \oplus \rho_2 \oplus \cdots \oplus 1^{m_{c-1}} \oplus \rho_c \oplus 1^{m_c}.
\]
In this case, we define the \emph{rigid reduction} of $\pi$
\[
\rigidReduction (\pi) = \rho_1  \oplus \rho_2 \oplus \cdots  \oplus \rho_c.
\]
For example:
\[
\rigidReduction (2413\,5\,76\,89) = 2413 \, 65.
\]
For a set $X$ of permutations $\rigidReduction (X) = \{ \rigidReduction(\pi) \, : \, \pi \in X\}$.

\section{The main result}\label{main section}

We now turn our attention almost exclusively to infinite subclasses of $\I_2 = \avoid(321)$ with the aim of proving:

\begin{theorem}
\label{THM Main Theorem}
Let $X$ be any finite set of permutations. Then $\I_2 \cap \avoid(X)$ and $\I_2 \cap \avoid(\rigidReduction(X))$ have the same growth rate.
\end{theorem}

This seems a surprising result as, \emph{a priori}, the second class appears to be much smaller than the first one -- consider for instance $\I_2 \cap \avoid(21 \, 34 \, 65 \, 7)$ and $\I_2 \cap \avoid(2143)$. To prove it, some further preparation is required.

A \emph{staircase decomposition} of a permutation $\pi \in \I_2$ is a partition $\alpha_1$, $\alpha_2$, \ldots, $\alpha_k$ of $\pi$ that has the following properties:
\begin{itemize}
\item
The pattern of each $\alpha_i$ is increasing;
\item
For $j \geq 1$, $\alpha_{2j}$ lies entirely to the right of $\alpha_{2j-1}$;
\item
For $j \geq 1$, $\alpha_{2j+1}$ lies entirely above $\alpha_{2j}$;
\item
If $i - j \geq 2$ then $\alpha_i$ lies entirely above and to the right of $\alpha_j$.
\end{itemize}
Figure \ref{FIG Staircase} should make it clear why the term ``staircase decomposition'' was chosen. We refer to the individual constituents $\alpha_i$ of the staircase as its \emph{blocks}.

\begin{figure}[ht]
\begin{center}
\begin{tabular}{ccc}
\psset{xunit=0.015in, yunit=0.015in}
\psset{linewidth=0.01in}
\begin{pspicture}(0,0)(100,100)
\psline(5,5)(65,5)(65,95)(95,95)(95,35)(5,35)(5,5)
\psline(35,5)(35,65)(95,65)
\pscircle*(10,15){0.02in}
\pscircle*(20,18){0.02in}
\pscircle*(30,30){0.02in}
\pscircle*(40,10){0.02in}
\pscircle*(60,25){0.02in}
\pscircle*(45,38){0.02in}
\pscircle*(50,48){0.02in}
\pscircle*(55,58){0.02in}
\pscircle*(68,40){0.02in}
\pscircle*(75,45){0.02in}
\pscircle*(82,53){0.02in}
\pscircle*(89,63){0.02in}
\pscircle*(85,80){0.02in}
\end{pspicture}
&\rule{40pt}{0pt}&
\psset{xunit=0.015in, yunit=0.015in}
\psset{linewidth=0.01in}
\begin{pspicture}(0,0)(100,100)
\psline(5,5)(65,5)(65,95)(95,95)(95,35)(5,35)(5,5)
\psline(35,5)(35,65)(95,65)
\pscircle*(10,10){0.02in}
\pscircle*(20,20){0.02in}
\pscircle*(30,30){0.02in}
\pscircle*(40,40){0.02in}
\pscircle*(50,50){0.02in}
\pscircle*(60,60){0.02in}
\pscircle*(70,70){0.02in}
\pscircle*(80,80){0.02in}
\pscircle*(90,90){0.02in}
\pscircle*(43,8){0.02in}
\pscircle*(53,18){0.02in}
\pscircle*(63,28){0.02in}
\pscircle*(73,38){0.02in}
\pscircle*(83,48){0.02in}
\pscircle*(93,58){0.02in}
\end{pspicture}
\end{tabular}
\end{center}
\caption{On the left, a staircase decomposition; and on the right, a generic staircase with five blocks of size three.}
\label{FIG Staircase}
\end{figure}
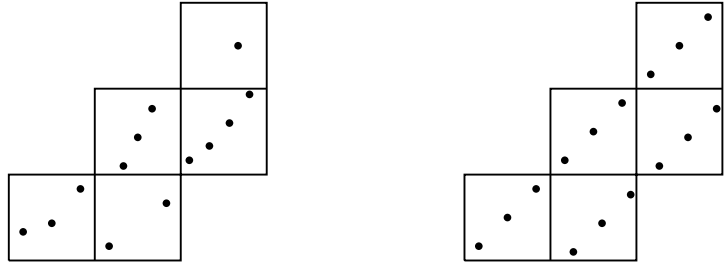

Every $\pi \in \I_2$ has a staircase decomposition. This can be constructed inductively by taking, for odd $i$, $\alpha_i$ to be the longest initial segment by position of $\displaystyle \pi \setminus \cup_{j < i} \alpha_j$ that has an increasing pattern; and for even $i$, $\alpha_i$ to be the longest initial segment by value of $\displaystyle \pi \setminus \cup_{j < i} \alpha_j$ that has an increasing pattern.

Let positive integers $k$ and $b$ be given. The \emph{generic staircase with $k$ blocks of size $b$} or $(k,b)$-generic staircase is that permutation $\pi$ which has a staircase decomposition $\alpha_1$, $\alpha_2$, \ldots, $\alpha_k$, where for each $i$, $|\alpha_i| = b$ and additionally:
\begin{itemize}
\item
If $i \geq 1$ and $t \leq b$, then the $t^{\mbox{\scriptsize th}}$ element of $\alpha_{2i}$ lies in value between the $(t-1)^{\mbox{\scriptsize st}}$ and $t^{\mbox{\scriptsize th}}$ elements of $\alpha_{2i-1}$;
\item
If $i \geq 1$ and $t \leq b$, then the $t^{\mbox{\scriptsize th}}$ element of $\alpha_{2i+1}$ lies in position between the $t^{\mbox{\scriptsize th}}$ and $(t+1)^{\mbox{\scriptsize st}}$ elements of $\alpha_{2i}$.
\end{itemize}
Figure \ref{FIG Staircase} also illustrates an example of a generic staircase.

\begin{proposition}
Every $\pi \in \I_2$ occurs as a pattern in some generic staircase.
\end{proposition}

\begin{proof}
Let $\pi \in \I_2$ be given, and choose a staircase decomposition $\alpha_1$, $\alpha_2$, \ldots, $\alpha_k$ of $\pi$. Consider the infinite set of points shown in Figure \ref{FIG Infinite Staircase}. The points in each of the line segments within a block are a translation of the set $D \cap (0,1)$ where $D$ is the set of dyadic rationals (rationals whose denominator is a power of 2) and therefore form a dense linear order without endpoints. Choose an arbitrary embedding of $\alpha_1$ into the first block. Then, $\alpha_2$ can be embedded into the second block in such a way that the pattern of $\alpha_1 \cup \alpha_2$ is preserved (simply because we have a dense linear order available here). Similarly, having embedded $\alpha_1$ and $\alpha_2$, we can embed $\alpha_3$ in the third block. Its relationship with the embedded copy of $\alpha_1$ is fixed by the fourth condition in the definition of a staircase decomposition, and its proper relationship with the embedded copy of $\alpha_2$ can be assured using the density again. Proceeding inductively we can find an embedding of $\pi$ into this infinite set. Since $\pi$ is finite, the range of this embedding is contained entirely among the points whose coordinates have a denominator at most $2^m$ for some $m$. Now reduce the infinite staircase to the finite set of points of this type. The result is not a generic staircase as some points share a common horizontal or vertical component. However, each odd numbered block can be shifted upwards by $1/2^{m+1}$ (or any suitably small amount) and each even numbered block leftwards by the same amount. This does not change the relationship of any pair of points that were previously on different horizontal or vertical lines (and in particular, the images of the points of $\pi$), and the resulting staircase is generic with $k$ blocks of size $2^m - 1$.

\begin{figure}[ht]
\begin{center}
\psset{xunit=0.02in, yunit=0.02in}
\psset{linewidth=0.01in}
\begin{pspicture}(0,0)(100,100)
\psline(5,5)(35,5)(35,35)(50,35)(50,20)(5,20)(5,5)
\psline(20,5)(20,35)(35,35)
\psline(65,50)(95,50)(95,80)(80,80)(80,50)
\psline(65,50)(65,65)(95,65)
\pscircle*(48,38){0.015in}
\pscircle*(55,42){0.015in}
\pscircle*(62,46){0.015in}
\psline(7,7)(18,18)
\pscircle[fillstyle=solid,fillcolor=white](7,7){0.015in}
\pscircle[fillstyle=solid,fillcolor=white](18,18){0.015in}
\psline(22,7)(33,18)
\pscircle[fillstyle=solid,fillcolor=white](22,7){0.015in}
\pscircle[fillstyle=solid,fillcolor=white](33,18){0.015in}
\psline(22,22)(33,33)
\pscircle[fillstyle=solid,fillcolor=white](22,22){0.015in}
\pscircle[fillstyle=solid,fillcolor=white](33,33){0.015in}
\psline(37,22)(48,33)
\pscircle[fillstyle=solid,fillcolor=white](37,22){0.015in}
\pscircle[fillstyle=solid,fillcolor=white](48,33){0.015in}
\psline(67,52)(78,63)
\pscircle[fillstyle=solid,fillcolor=white](67,52){0.015in}
\pscircle[fillstyle=solid,fillcolor=white](78,63){0.015in}
\psline(82,52)(93,63)
\pscircle[fillstyle=solid,fillcolor=white](82,52){0.015in}
\pscircle[fillstyle=solid,fillcolor=white](93,63){0.015in}
\psline(82,67)(93,78)
\pscircle[fillstyle=solid,fillcolor=white](82,67){0.015in}
\pscircle[fillstyle=solid,fillcolor=white](93,78){0.015in}
\end{pspicture}
\end{center}
\caption{A staircase where each block is a dense linear order without endpoints.}
\label{FIG Infinite Staircase}
\end{figure}
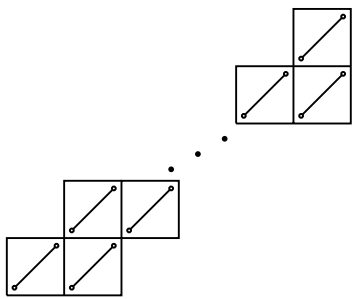
\end{proof}


The following technical proposition links together bounded merges  and generic staircases.  It shows that a $321$-avoiding permutation that avoids a generic staircase is a bounded merge of two increasing permutations where the parameters of the bounded merge are dependent on the parameters of the generic staircase.  We use it in Propositions \ref{Special_Case} and \ref{General_Case} to show that a permutation of $\avoid(321)$ that avoids some extra pattern other than $321$ lies in a bounded merge of classes which avoid shorter (but related) patterns.

\begin{proposition}
\label{PROP Staircase or Merge}
Let positive integers $k$ and $b$ be given. There is a positive integer $B$ (depending only on $k$ and $b$) such that for all $\pi \in \I_2$, either $\pi$ contains a $(k,b)$-generic staircase, or $\pi$ is a $B$-bounded merge of two permutations $\lambda$ and $\beta$ such that the image of $\lambda$ contains all the elements preceding the minimum element of $\pi$, and the image of $\beta$ contains all the elements less than the first element of $\pi$.
\end{proposition}

\begin{proof}
The proof will show that the proposition is true with $B=(k+2)(b+1)/2$. 

Let $\pi \in \I_2$ be given. Then there is a decomposition of $\pi$ into a pair of intertwined staircases which is illustrated in Figure \ref{FIG Intertwined Staircases}. In this decomposition consider the staircase that begins with the block $\lambda_1$ which consists of all the elements preceding the least element of $\pi$. If this staircase has fewer than $k$ blocks then $\pi$ is a $k$-bounded merge of two permutations having the requisite properties. So, suppose that at least $k$ blocks occur in this staircase.
\begin{figure}[ht]
\begin{center}
\psset{xunit=0.02in, yunit=0.02in}
\psset{linewidth=0.01in}
\begin{pspicture}(0,0)(100,100)
\rput(5,20){%
\psline(0,0)(15,0)(15,15)(0,15)(0,0)
\psline(3,3)(12,12)
}
\rput(35,20){%
\psline(0,0)(15,0)(15,15)(0,15)(0,0)
\psline(3,3)(12,12)
}
\rput(35,50){%
\psline(0,0)(15,0)(15,15)(0,15)(0,0)
\psline(3,3)(12,12)
}
\rput(65,50){%
\psline(0,0)(15,0)(15,15)(0,15)(0,0)
\psline(3,3)(12,12)
}
\rput(65,80){%
\psline(0,0)(15,0)(15,15)(0,15)(0,0)
\psline(3,3)(12,12)
}
\rput(20,5){%
\psline[linestyle=dashed, dash=3pt 2pt](0,0)(15,0)(15,15)(0,15)(0,0)
\psline[linestyle=dotted,dotsep=1.5pt](3,3)(12,12)
}
\rput(20,35){%
\psline[linestyle=dashed, dash=3pt 2pt](0,0)(15,0)(15,15)(0,15)(0,0)
\psline[linestyle=dotted,dotsep=1.5pt](3,3)(12,12)
}
\rput(50,35){%
\psline[linestyle=dashed, dash=3pt 2pt](0,0)(15,0)(15,15)(0,15)(0,0)
\psline[linestyle=dotted,dotsep=1.5pt](3,3)(12,12)
}
\rput(50,65){%
\psline[linestyle=dashed, dash=3pt 2pt](0,0)(15,0)(15,15)(0,15)(0,0)
\psline[linestyle=dotted,dotsep=1.5pt](3,3)(12,12)
}
\rput(80,65){%
\psline[linestyle=dashed, dash=3pt 2pt](0,0)(15,0)(15,15)(0,15)(0,0)
\psline[linestyle=dotted,dotsep=1.5pt](3,3)(12,12)
}
\end{pspicture}
\end{center}
\caption{A general picture of intertwined staircases. The solid blocks represent $\lambda_1$, $\lambda_2$ etc.}
\label{FIG Intertwined Staircases}
\end{figure}
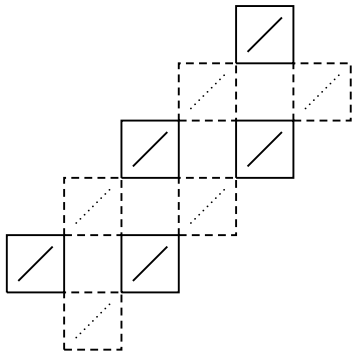

Label the elements of these blocks in the following way:
\begin{itemize}
\item
The elements of $\lambda_1$ are labeled with their values.
\item
For even $i > 1$, each element of $\lambda_i$ is labeled with the largest label of an element of $\lambda_{i-1}$ of smaller value.
\item
For odd $i > 1$, each element of $\lambda_i$ is labeled with the largest label of an element of $\lambda_{i-1}$ to its left.
\end{itemize}
Note that, within each block, if a label occurs in that block, then it labels an interval of elements in the block; and that together with all the elements of the preceding block sharing the same label we obtain an interval by position or value within $\pi$ according to whether the block is of odd or even index.

Our first claim is that if at least $b$ labels occur in $\lambda_k$, then $\pi$ contains a $(k,b)$-generic staircase. This is clear enough: simply choose a set of $b$ labels that occur in $\lambda_k$ and then, for each chosen label, in each $\lambda_i$ for $1 \leq i \leq k$ take the first element carrying that label. The pattern of these elements is that of a $(k,b)$-generic staircase.

So, we assume henceforth that the set $L$ of labels occurring in $\lambda_k$ has fewer than $b$ elements. Let $C$ be its complement (in the set of labels occurring in $\lambda_1$). We claim that if we take $\lambda$ to consist of all elements with labels in $C$ together with all the elements of $\lambda_1$, and take $\beta$ to be the remaining elements of $\pi$, then the number of alternations between $\lambda$ and $\beta$ in the resulting merge is bounded by a function of $k$ and $b$ (independent of $\pi$). 

Consider the elements of $\lambda_1$ through $\lambda_k$ whose labels come from $C$ (there are of course none in $\lambda_k$). They define a certain set of intervals by value and by position in $\pi$. If $x, y \in \lambda_i$ lie in different intervals, then they are separated by an element whose label is in $L$.  Thus, using the note following the definition of labeling, the elements of $C$ belonging to a vertical pair of blocks ($\lambda_{2i}$ and $\lambda_{2i+1}$) project onto at most $|L| + 1$ intervals by position. Similarly, the elements of $C$ in a horizontal pair of blocks project onto at most $|L|+1$ intervals by value. So, within $\pi$ the number of intervals determined by the elements with labels from $C$ is bounded above by $k(|L|+1)/2$ (whether we consider intervals by position or by value). Now add to this set of elements the remaining $|L|$ elements of $\lambda_1$. This might increase the number of intervals by value, but not by more than the number of elements added. If anything, it decreases the number of intervals by position (since the entire block $\lambda_1$ is now included which forms a single interval by position). So, $\lambda_1$ together with elements whose labels come from $C$ determine at most $k(b+1)/2 + b$ intervals either by position or value. We set $\lambda$ to be the pattern of this part, $\beta$ the pattern of the remainder of $\pi$ and then their merge has at most $1+k(b+1)/2 + b$ type changes.
\end{proof}

We have all the tools required to prove Theorem \ref{THM Main Theorem} at this point, but it will still be helpful to approach it gently. The following proposition is not technically required in the main proof, but isolates half of the argument and, we hope, will make it easier to follow the full proof. It is also included for historical accuracy -- this result was proved before the significance of rigid permutations in the main result was understood.

\begin{proposition}\label{Special_Case}
Let $X \subseteq \I_2$, $\beta \in \I_2 \cap \avoid(X)$ and suppose that $\C = \I_2 \cap \avoid(X) \cap \avoid(\beta)$ is an infinite class. Then, the growth rates of $\C$ and $\C' = \I_2 \cap \avoid(X) \cap \avoid(1 \oplus \beta)$ are the same.
\end{proposition}

\begin{proof}
Since $\C \subseteq \C'$ it is sufficient to show that $\C' \setminus \C$ is contained in some class (or indeed any set) whose growth rate is not greater than that of $\C$. So, let $\pi \in \C' \setminus \C$ be given. If $\pi$ begins with its minimum, then it belongs to the class $\C \cup (1 \oplus \C)$ and this class has the same growth rate as $\C$ does. Otherwise, since $\pi$ avoids $1 \oplus \beta$, and hence also some generic staircase, it must by Proposition \ref{PROP Staircase or Merge} be a bounded merge of two permutations each avoiding $1 \oplus \beta$ and each beginning with their minimum elements. Since these permutations avoid $1 \oplus \beta$, their patterns after the first element must avoid $\beta$. So, in any case, $\pi$ belongs to a bounded merge of the class $1 \oplus \C$ with itself. Thus $\stanleyWilfLimit (\C) = \stanleyWilfLimit (\C')$ as claimed.
\end{proof}

Now we extend this proposition to a form from which Theorem \ref{THM Main Theorem} will follow by an easy inductive argument.

\begin{proposition}\label{General_Case}
Let $X \subseteq \I_2$, $\alpha, \beta \in \I_2$ and suppose that $\alpha$ is 2-rigid, $\alpha \oplus \beta \in \I_2 \cap \avoid(X)$ and  $\C = \I_2 \cap \avoid(X) \cap \avoid(\alpha \oplus \beta)$ is an infinite class. Then, the growth rates of $\C$ and $\C' = \I_2 \cap \avoid(X) \cap \avoid(\alpha \oplus 1 \oplus \beta)$ are the same.
\end{proposition}

\begin{proof}
We proceed as in the previous proposition. Let $\pi \in \C' \setminus \C$. Since $\pi$ contains an embedded copy of $\alpha \oplus \beta$, it contains such a copy in which the $\alpha$ pattern is witnessed by the leftmost-bottommost copy of $\alpha$ in $\pi$ (whose existence is assured by Corollary \ref{COR leftmost-bottommost}). The general disposition of $\pi$ is then as shown in Figure \ref{FIG bigPic}.

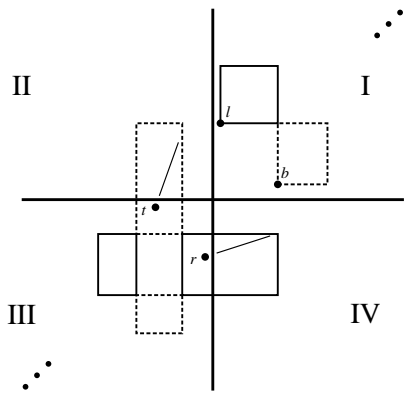
\begin{figure}[ht]
\begin{center}
\psset{xunit=0.02in, yunit=0.02in}
\psset{linewidth=0.015in}
\begin{pspicture}(0,0)(100,100)
\psline(0,50)(100,50)
\psline(50,0)(50,100)
\psset{linewidth=0.01in}
\pscircle*(1,1){0.015in}
\pscircle*(4,4){0.015in}
\pscircle*(7,7){0.015in}
\pscircle*(93,93){0.015in}
\pscircle*(96,96){0.015in}
\pscircle*(99,99){0.015in}
\pscircle*(35,48){0.02in} 
\pscircle*(48,35){0.02in} 
\pscircle*(52,70){0.02in} 
\pscircle*(67,54){0.02in} 
\psline[linestyle=dashed, dash=1.5pt 0.9pt](30,41)(42,41)(42,70)(30,70)(30,41) 
\psline(52,70)(67,70)(67,85)(52,85)(52,70) 
\psline(42,41)(67,41)(67,25)(42,25)(42,41) 
\psline[linestyle=dashed, dash=1.5pt 0.9pt](67,54)(80,54)(80,70)(67,70)(67,54) 
\psline(30,41)(30,25)(20,25)(20,41)(30,41) 
\psline[linestyle=dashed, dash=1.5pt 0.9pt](42,25)(42,15)(30,15)(30,25)(42,25) 
\psset{linewidth=0.006in}
\psline(36,51)(41,65) 
\psline(51,36)(65,40.5) 
{\small
\rput[c](90,80){I}
\rput[c](0,80){II}
\rput[c](0,20){III}
\rput[c](90,20){IV}
}
{\tiny
\rput[c](32,47){$t$}
\rput[c](45,34){$r$}
\rput[c](54,73){$l$}
\rput[c](69,57){$b$}
}
\end{pspicture}
\end{center}
\caption{The structure of $\pi$ containing $\alpha \oplus \beta$. The leftmost-bottommost $\alpha$ is contained in quadrant III. Its maximum is the element $t$ and its rightmost element $r$. Quadrant I with leftmost element $l$ and minimum $b$, contains a copy of $\beta$. All of $\pi$ can be represented as a bounded merge of two permutations, one part of which contains the solid boxes and the other the dotted boxes.}
\label{FIG bigPic}
\end{figure}

If quadrant I began with its minimum element, we could represent $\pi$ as the merge of two permutations -- that singleton element, and the rest. Those remaining elements would have to avoid the pattern $\alpha \oplus \beta$ as otherwise using the leftmost-bottomost $\alpha$, the singleton element, and any copy of $\beta$ which is part of an $\alpha \oplus \beta$ we would have $\alpha \oplus 1 \oplus \beta \involvedIn \pi$. So we may assume that the situation is as shown in the figure, that is that the leftmost element of quadrant I (marked $l$) and its minimum (marked $b$) are distinct.

As before, we can decompose quadrant I containing $\beta$ into a pair of intertwined staircases, and thus represent it as a bounded merge of two permutations (since it must avoid $1 \oplus \beta$ and hence some generic staircase). The remainder of the permutation consists of the part in quadrant III bounded by the topmost and rightmost points of the copy of $\alpha$, together with two increasing segments (either or both of which may be empty) in quadrants II and IV as shown. This subset of $\pi$ must avoid $\alpha \oplus 1 \oplus \beta$ and so can also be written as a bounded merge of two permutations, one containing the solid rectangle to which $r$ belongs, and the other the dotted rectangle to which $t$ belongs, as shown in the figure. Here we use Proposition \ref{PROP Staircase or Merge} applied to the pattern of these elements obtained by a $180^\circ$ degree rotation of the graph.

These two bounded merges can be combined into a single bounded merge which represents the entire permutation $\pi$. We will now show that neither of the components of this merge contains a copy of $\alpha \oplus \beta$. Suppose, for the sake of argument, that the component, $\sigma$, represented by the solid boxes contained this pattern, on a subset $\theta$ containing the leftmost-bottommost copy of $\alpha$ in $\sigma$. The leftmost-bottommost copy of $\alpha$ in $\sigma$ would extend strictly above the leftmost-bottommost copy of $\alpha$ in $\pi$, since $\sigma$ does not contain the topmost element ($t$) of the leftmost-bottommost copy of $\alpha$ in $\pi$. So, the copy, $\beta'$, of $\beta$ in $\theta$ lying above this copy of $\alpha$ could not include the leftmost element ($l$) of quadrant I; as all the elements of  $\pi$ larger than $t$ either lie in the other part of the merge, or properly within quadrant I. Therefore, $\beta'$ lies strictly above and to the right of $l$. However, $\alpha'$, the leftmost-bottommost copy of $\alpha$ in $\pi$ lies strictly below and to the left of $l$. In that case the pattern of $\alpha' \cup \{l\} \cup \beta'$ is $\alpha \oplus 1 \oplus \beta$, providing a contradiction as $\pi$ avoids this permutation. The argument that the other part of the merge cannot contain $\alpha \oplus \beta$ is similar.

Hence, any element of $\C' \setminus \C$ is a bounded merge of two permutations in $\C$ and thus the growth rates of $\C'$ and $\C$ are the same.
\end{proof}

Now finally:

\begin{proof-main}
Without loss in generality we may assume that $X\subseteq \I_2$.  Furthermore we may assume that $X$ does not contain any increasing permutation and so the class $\I_2 \cap \avoid(X)$ is infinite (the result is, of course, trivial if this class is finite). If $\rigidReduction(X) = X$ there is nothing to prove. Otherwise, $X$ contains at least one permutation, $\tau$, having an articulation point. Write $\tau = \alpha \oplus 1 \oplus \beta$ where $\alpha$ is either rigid or empty (that is, decompose $\tau$ around its first articulation point). Let $\tau' = \alpha \oplus \beta$ and $X' = (X \setminus \{ \tau \}) \cup \{ \tau' \}$. Then, by one of the two preceding propositions
\[
\stanleyWilfLimit(\I_2 \cap \avoid(X)) = \stanleyWilfLimit(\I_2 \cap \avoid(X')).
\]
After a series of such reductions (formally, by induction on the number of articulation points occurring among the elements of $X$) we obtain the desired conclusion.
\end{proof-main}

\section{The lattice of embeddings of \texorpdfstring{\boldmath $21$}{21} in an element of \texorpdfstring{\boldmath $\I_2$}{I2}}
\label{aftermath}

Theorem \ref{THM Lattice Structure} showed that the embeddings of a $k$-rigid permutation $\rho$ into an element of $\I_k$ form a distributive lattice. The case $k = 2$, and $\rho = 21$ is particularly interesting. The union of the images of $21$ in a permutation $\pi \in \I_2$ forms exactly the rigid reduction of $\pi$, so we interest ourselves only in the case where $\pi$ is $2$-rigid, and we set $L_{\pi}$ to be the distributive lattice of copies of $21$ in $\pi$. Restricting further, we consider as fixed the number, $m$, of rank 2 elements in $\pi$ and also the number, $n$ of rank $1$ elements, and we represent these by the chain $[m] = \{1,2,3,\ldots,m\}$ and $[n] = \{1,2,3,\ldots,n\}$ respectively. We suppress a necessary distinction between these chains according to the rank of the corresponding elements, since this is always clear from context.
Then $L_{\pi}$ forms a sublattice of $[m] \times [n]$, where $(i,j) \in L_\pi$ if and only if the $i^{\mbox{\scriptsize th}}$ element of rank $2$ and the $j^{\mbox{\scriptsize th}}$ element of rank $1$ form a $21$-pattern. In particular, if $\pi= (n+1)\cdots (n+m) 1\cdots n$, then $L_{\pi} = [m] \times [n]$. Another example is shown in Figure \ref{FIG PermAndLattice}.

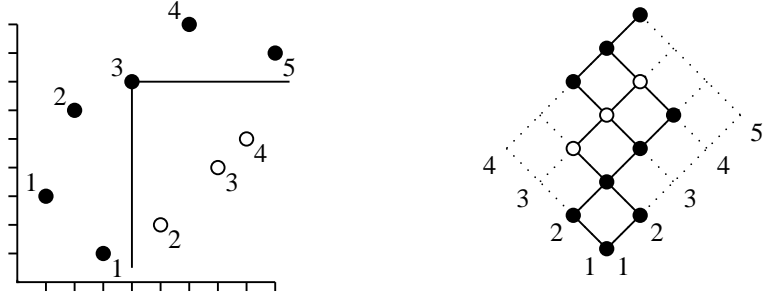
\begin{figure}
\begin{center}
\begin{tabular}{ccc}
\psset{xunit=0.015in, yunit=0.015in}
\psset{linewidth=0.01in}
\begin{pspicture}(0,0)(100,100)
\psaxes[dy=10,Dy=1,dx=10,Dx=1,tickstyle=bottom,showorigin=false,labels=none](0,0)(90,90)
\pscircle*(10,30){0.04in}
\pscircle*(20,60){0.04in}
\pscircle*(30,10){0.04in}
\pscircle*(40,70){0.04in}
\pscircle[fillstyle=solid,fillcolor=white](50,20){0.04in}
\pscircle*(60,90){0.04in}
\pscircle[fillstyle=solid,fillcolor=white](70,40){0.04in}
\pscircle[fillstyle=solid,fillcolor=white](80,50){0.04in}
\pscircle*(90,80){0.04in}
{\footnotesize %
\rput[c](5,35){$1$}
\rput[c](15,65){$2$}
\rput[c](35,75){$3$}
\rput[c](55,95){$4$}
\rput[c](35,5){$1$}
\rput[c](55,15){$2$}
\rput[c](75,35){$3$}
\rput[c](85,45){$4$}
\rput[c](95,75){$5$}}
\psline(40,5)(40,70)(95,70)
\end{pspicture}
&\rule{40pt}{0pt}&
\psset{xunit=0.0175in, yunit=0.0175in}
\psset{linewidth=0.01in}
\begin{pspicture}(-10,-10)(80,80)
\psline[linestyle=dotted](0,30)(40,70)
\psline[linestyle=dotted](10,20)(50,60)
\psline[linestyle=dotted](20,10)(60,50)
\psline[linestyle=dotted](30,0)(70,40)
\psline[linestyle=dotted](30,0)(0,30)
\psline[linestyle=dotted](40,10)(10,40)
\psline[linestyle=dotted](50,20)(20,50)
\psline[linestyle=dotted](60,30)(30,60)
\psline[linestyle=dotted](70,40)(40,70)
\psline(30,0)(40,10)
\psline(20,10)(50,40)
\psline(20,30)(40,50)
\psline(20,50)(40,70)
\psline(30,0)(20,10)
\psline(40,10)(20,30)
\psline(40,30)(20,50)
\psline(50,40)(30,60)
\pscircle*(30,0){0.04in}
\pscircle*(20,10){0.04in}
\pscircle*(40,10){0.04in}
\pscircle*(30,20){0.04in}
\pscircle[fillstyle=solid,fillcolor=white](20,30){0.04in}
\pscircle*(40,30){0.04in}
\pscircle[fillstyle=solid,fillcolor=white](30,40){0.04in}
\pscircle*(50,40){0.04in}
\pscircle*(20,50){0.04in}
\pscircle[fillstyle=solid,fillcolor=white](40,50){0.04in}
\pscircle*(30,60){0.04in}
\pscircle*(40,70){0.04in}
{\footnotesize %
\rput[c](25,-5){$1$}
\rput[c](15,5){$2$}
\rput[c](5,15){$3$}
\rput[c](-5,25){$4$}
\rput[c](35,-5){$1$}
\rput[c](45,5){$2$}
\rput[c](55,15){$3$}
\rput[c](65,25){$4$}
\rput[c](75,35){$5$}}
\end{pspicture}
\end{tabular}
\end{center}
\caption{The permutation $\pi=361729458$ and its corresponding lattice, with the interval $D(3)=\{2,3,4\}$ of rank $1$ points highlighted.}
\label{FIG PermAndLattice}
\end{figure}

Recall that if $A$ and $B$ are algebraic structures, then a subalgebra $C \leq A \times B$ is called a \emph{subdirect product} of $A$ and $B$ if the projections from $C$ to $A$ and to $B$ are both surjective. The lattice $L_\pi$ is always a subdirect product of $[m]$ and $[n]$ since every element is part of some $21$. Also it is clear that if $\pi \neq \pi'$, then $L_{\pi} \neq L_{\pi'}$, since all the order relationships of $\pi$ are determined by $L_{\pi}$.

Now suppose that $K$ is an arbitrary subdirect product of $[m]$ and $[n]$. For $a \in [m]$ define $D_K(a) = \{ p \in [n] \,: \, (a,p) \in K\}$. The following observation is certainly folkloric:

\begin{observation}
For all $a \in [m]$, $D_K(a)$ is a non-empty interval in $[n]$. Furthermore if $a, b \in [m]$ with $a < b$ then $\min D_K(a) \leq \min D_K(b)$ and $\max D_K (a) \leq \max D_K(b)$.
\end{observation}

\begin{proof}
For the first part, suppose that $p \leq q \leq r$ and $p, r \in D_k(a)$. Then, $(b,q) \in K$ for some $b$, because $K$ is subdirect. If $b \leq a$ then $(a,q) = (b, q) \vee (a, p)$, while if $ a < b$ then $(a,q) = (b,q) \wedge (a,r)$. In either case, $q \in D_K(a)$. The second part is immediate as well, for if $(a, p) \in K$ and $(b, q) \in K$ then $(a, p \wedge q), (b, p \vee q) \in K$.
\end{proof}

Using this observation we can construct, given $K$, a permutation $\Pi(K)$ as follows: begin with an increasing sequence of length $n$. Now, for $a \in [m]$ place a new element just to the left of $\min D(a)$ and just above $\max D(a)$ (and also above all previously placed elements of this sort). The conditions of the observation guarantee that such a placement is always possible. It is also clear that $L_{\Pi(K)} = K$. Thus we obtain:

\begin{theorem}
The 2-rigid elements of $\I_2$ having $m$ elements of rank 2 and $n$ elements of rank $1$ are in one-to-one correspondence with the subdirect products of $[m]$ and $[n]$.
\end{theorem}

\begin{proof}
We have noted that the association $\pi \mapsto L_{\pi}$ is both one-to-one and onto.
\end{proof}

Releasing the restrictions on $m$ and $n$ we see that every subdirect product of two finite chains is equal to $L_{\pi}$ for a unique 2-rigid permutation $\pi \in \I_2$. However, for 3-rigid permutations in $\I_3$ no such result holds. For example, there are 29 subdirect products of three chains of length 2, but only 25 permutations that are 3-rigid of size 6 with 2 elements of each rank. In fact, even among these permutations there are duplications in their corresponding lattices. 

A permutation is 2-rigid if it is covered by its embedded copies of 21. We noted above that we could count the number of 2-rigid permutations in $\I_2$ and we might well consider what we can say about permutations satisfying some stronger conditions. For example, we might call $\pi \in \I_2$ \emph{$k$-good} if every point of $\pi$ lies in a copy of $\iota_k \ominus \iota_k$. Thus, a 1-good permutation is 2-rigid, and vice versa. We do not have a complete enumeration of this collection of permutations, but the following result is amusing:

\begin{lemma}
\label{lem-k-good-perms}
There are $\binom{2\ell}{\ell}$ $k$-good permutations of length $2k+\ell$ for $0\leq \ell \leq k$.
\end{lemma}

\begin{proof}
Let $a_j$ denote the number of $k$-good permutations of length $2k+\ell$ for which there are $k+j$ points lying on the upper line (and subsequently $k+\ell-j$ on the lower). Note first that $a_j=0$ for every $j>\ell$, as then there are fewer than $k$ points on the lower line. Thus we need only consider values of $j$ satisfying $0\leq j\leq \ell$. 

Supposing $\pi$ is such a permutation, divide each line into three sections: from left to right, the upper line is divided into (possibly empty) parts of sizes $j$, $k-j$ and $j$, and the lower into $\ell -j$, $k-\ell+j$ and $\ell-j$. Note that the condition $\ell\leq k$ ensures that this division is possible. Since $\pi$ is $k$-good, the middle sections of each line (of sizes $k-j$ and $k-\ell+j$) cannot interact: the leftmost $k$ points of each of the upper and lower lines must together form a copy of $\iota_k\ominus \iota_k$, and so the middle section of each line cannot interact with the leftmost section of the other. Similarly, the rightmost $k$ points of each line must also form an $\iota_k\ominus\iota_k$ , and hence the middle section of each line cannot interact with the rightmost section of the other. Trivially, these two conditions also prevent the middle sections from interacting with each other.

Thus $a_j$ counts the number of ways of simultaneously interleaving the rightmost part of the upper line with the leftmost part of the lower line vertically, and the leftmost of the upper with the rightmost of the lower horizontally. Up to symmetry these two interleavings are the same, so we consider only the former. Note that these two sections contain a total of $\ell$ points, and so there are $\binom{\ell}{j}$ possible interleavings. Hence $a_j=\binom{\ell}{j}^2$, and so there are $\sum_{j=0}^\ell \binom{\ell}{j}^2 = \binom{2\ell}{\ell}$ such permutations.
\end{proof}

It is worth noting that there are also $\binom{2\ell}{\ell}$ $k$-good permutations of length $2k+\ell$ when $\ell=k+1$: the argument in the proof of Lemma~\ref{lem-k-good-perms} still works for $j$ satisfying $1\leq j\leq k$ (i.e.\ $a_j=\binom{\ell}{j}^2$). When $j=0$, the upper line contains exactly $k$ points and there is only one such $k$-good permutation of each length of this form, giving $a_0=1$. Similarly, when $j=\ell=k+1$ the lower line contains exactly $k$ points, and again we always have $a_{k+1}=1$.

\section{Further remarks}

We have been unable to extend the main result of Section~\ref{main section} to apply to the classes $\I_k$ with $k\geq 3$. This is largely because there seems to be no analog to the ``generic staircase'' which we require in order to obtain bounded merges. Indeed, Waton's doctoral thesis \cite{Waton07} points to a fundamental difference between $\I_2$ and $\I_3$.  He considered their subclasses from a combinatorial-geometric point of view.  In his work $\I_2$ arises as the set of all permutations drawn on two fixed arbitrary parallel lines. By way of contrast, permutations drawn on three parallel lines form a proper subclass of $\I_3$, and there are uncountably many such classes, depending on the relative position of the three lines.  Despite this we have managed to prove a weaker form of the result  (generalizing an unpublished observation of M. B\'{o}na):

\begin{proposition}
For any $k$, $\alpha$ and $\beta$, and set of permutations $X$, the growth rates of $\I_k\cap\avoid(X, \alpha \oplus 1 \oplus \beta)$ and $\I_k\cap\avoid( X, \alpha \oplus 1 \oplus 1 \oplus \beta)$ are the same.
\end{proposition}
%

\begin{proof}
As usual, consider those $\pi \in \I_k$ which avoid $\alpha \oplus 1 \oplus 1 \oplus \beta$ but involve $\alpha \oplus 1 \oplus \beta$. Consider all the elements $x$ of $\pi$ which have an $\alpha$ below and to their left, and a $\beta$ above and to their right. No two of these can form a $12$ pattern or else we would obtain a copy of $\alpha \oplus 1 \oplus 1 \oplus \beta$. Thus they form a descending chain, but in particular there can be at most $k$ of them. So $\pi$ is a bounded merge of a permutation avoiding $\alpha \oplus 1 \oplus \beta$ (as well as $\delta_{k+1}$) and a permutation of length at most $k$, which is all that we require.
\end{proof}

Applying this proposition repeatedly we can partially reduce the elements of any basis set of a class of this type without changing its growth rate, where by \emph{partial reduction} we mean replacing multiple consecutive articulation points by a single one.

As is well known, the class $\I_2$ is enumerated by the Catalan numbers. If we denote its generating function by $c$, and let $r$ denote the generating function of the rigid permutations in $\I_2$ (including the empty permutation), then the decomposition of an arbitrary $\pi \in \I_2$ used to define the rigid reduction shows that:
\[
c = \frac{r}{1-tr}.
\]
Therefore, 
\begin{eqnarray*}
c &=& \frac{1 - \sqrt{1-4t}}{2t} \\
r &=& \frac{1 + 2t - \sqrt{1 - 4t}}{2t(t+2)}.
\end{eqnarray*}
Then, the elementary estimates referred to in Example {\bf IV.2} (page 228) of \cite{FS09} applied to both $c$ and to $r$ yield:
 
\begin{proposition}
Asymptotically, $4/9$ of the permutations in $\I_2$ are 2-rigid. 
\end{proposition}	

This provides rather slim grounds on which to make the following:

\begin{conjecture}
Asymptotically, a positive proportion of the permutations in $\I_k$ are $k$-rigid.
\end{conjecture}

\bibliographystyle{acm}
\bibliography{combinatorics}

\begin{thebibliography}{10}

\bibitem{AA03b}
{\sc Albert, M., and Atkinson, M.}
\newblock Simple permutations, partial well-order, and enumeration.
\newblock In {\em Proceedings of Permutation Patterns 2003\/} (2003), pp.~5--9.

\bibitem{AMR02a}
{\sc Atkinson, M.~D., Murphy, M.~M., and Ru{\v{s}}kuc, N.}
\newblock Partially well-ordered closed sets of permutations.
\newblock {\em Order 19}, 2 (2002), 101--113.

\bibitem{Bona05}
{\sc B{\'o}na, M.}
\newblock The limit of a {S}tanley-{W}ilf sequence is not always rational, and
  layered patterns beat monotone patterns.
\newblock {\em J. Combin. Theory Ser. A 110}, 2 (2005), 223--235.

\bibitem{Bona07}
{\sc B{\'o}na, M.}
\newblock New records in {S}tanley-{W}ilf limits.
\newblock {\em European J. Combin. 28}, 1 (2007), 75--85.

\bibitem{ES35}
{\sc Erd{\H{o}}s, P., and Szekeres, G.}
\newblock A combinatorial problem in geometry.
\newblock {\em Compos. Math. 2\/} (1935), 463--470.

\bibitem{FS09}
{\sc Flajolet, P., and Sedgewick, R.}
\newblock {\em Analytic Combinatorics}.
\newblock Cambridge University Press, The Edinburgh Building, Cambridge, 2009.

\bibitem{Knuth75}
{\sc Knuth, D.~E.}
\newblock {\em The art of computer programming}, second~ed.
\newblock Addison-Wesley Publishing Co., Reading, Mass.-London-Amsterdam, 1975.
\newblock Volume 1: Fundamental algorithms, Addison-Wesley Series in Computer
  Science and Information Processing.

\bibitem{MT04}
{\sc Marcus, A., and Tardos, G.}
\newblock Excluded permutation matrices and the {S}tanley-{W}ilf conjecture.
\newblock {\em J. Combin. Theory Ser. A 107}, 1 (2004), 153--160.

\bibitem{Regev}
{\sc Regev, A.}
\newblock Asymptotic values for degrees associated with strips of young
  diagrams.
\newblock {\em Adv. Math. 41\/} (1981), 115--136.

\bibitem{Vatter08}
{\sc Vatter, V.}
\newblock Permutation classes of every growth rate (a.k.a. {S}tanley-{W}ilf
  limit) above 2.48187..
\newblock {\tt arXiv:0807.2815v1 [math.CO]}.

\bibitem{Waton07}
{\sc Waton, S.~D.}
\newblock {\em On permutation classes defined by token passing networks,
  gridding matrices and pictures: three flavours of involvement}.
\newblock PhD thesis, University of St. Andrews, 2007.

\end{thebibliography}

\end{document}